\documentclass[12pt]{amsart}
\usepackage{amssymb,amsmath,amsfonts,latexsym,amsthm,geometry}
\usepackage{amsmath}
\usepackage{amsfonts}
\usepackage{amssymb}
\usepackage{graphicx}

\providecommand{\U}[1]{\protect\rule{.1in}{.1in}}
\providecommand{\U}[1]{\protect\rule{.1in}{.1in}}
\providecommand{\U}[1]{\protect\rule{.1in}{.1in}}
\newtheorem{theorem}{Theorem}[section]

\newtheorem{lemma}[theorem]{Lemma}

\theoremstyle{definition}

\begin{document}
\title[A note on the Bohnenblust--Hille inequality for multilinear forms]{A
note on the Bohnenblust--Hille inequality for multilinear forms}
\author[J. Santos]{J. Santos}
\address{Departamento de Matem\'{a}tica \\
Universidade Federal da Para\'{\i}ba \\
58.051-900 - Jo\~{a}o Pessoa, Brazil.}
\email{joedsonmat@gmail.com}
\author[T. Velanga]{T. Velanga}
\address{IMECC \\
UNICAMP-Universidade Estadual de Campinas \\
13.083-859 - S\~{a}o Paulo, Brazil.}
\address{Departamento de Matem\'{a}tica \\
Universidade Federal de Rond\^onia \\
76.801-059 - Porto Velho, Brazil.}
\email{thiagovelanga@gmail.com}
\thanks{J. Santos was supported by CNPq, grant 303122/2015-3}

\begin{abstract}
The general versions of the Bohnenblust--Hille inequality for $m$-linear
forms are valid for exponents $q_{1},...,q_{m}\in \lbrack 1,2].$ In this
paper we show that a slightly different characterization is valid for $%
q_{1},...,q_{m}\in (0,\infty ).$
\end{abstract}

\maketitle



\section{Introduction}

The Bohnenblust-Hille inequality \cite{bh} asserts that for all positive
integers $m\geq 1$ there is a constant $C=C(\mathbb{K}$,$m)\geq 1$ such that
\begin{equation}
\left( \sum\limits_{i_{1},...,i_{m}=1}^{\infty }\left\vert
T(e_{i_{^{1}}},...,e_{i_{m}})\right\vert ^{\frac{2m}{m+1}}\right) ^{\frac{m+1%
}{2m}}\leq C\left\Vert T\right\Vert  \label{juui}
\end{equation}%
for all continuous $m$-linear forms $T:c_{0}\times \cdots \times
c_{0}\rightarrow \mathbb{K}$, where $\mathbb{K}$ denotes the fields of real
or complex scalars. This result can be generalized in some different
directions. A very interesting and far reaching generalization is the following (below, $%
e_{_{j}}^{n_{j}}$ means $(e_{j},...,e_{j})$ repeated $n_{j}$ times):

\bigskip

\begin{theorem}[Albuquerque, Ara\'{u}jo, Nu\~{n}ez, Pellegrino and Rueda]
(\cite{aa})\label{809} Let $1\leq k\leq m$ and $n_{1},\ldots ,n_{k}\geq 1$
be positive integers such that $n_{1}+\cdots +n_{k}=m$, let $q_{1},\dots
,q_{k}\in \lbrack 1,2]$. The following assertions are equivalent :

(I) There is a constant $C_{k,q_{1}...q_{k}}^{\mathbb{K}}\geq 1$ such that%
{\small {\
\begin{equation*}
\left( {\sum\limits_{i_{1}=1}^{\infty }}\left( {\sum\limits_{i_{2}=1}^{%
\infty }}\left( ...\left( {\sum\limits_{i_{k-1}=1}^{\infty }}\left( {%
\sum\limits_{i_{k}=1}^{\infty }}\left\vert T\left( e_{i_{1}}^{n_{1}},\ldots
,e_{i_{k}}^{n_{k}}\right) \right\vert ^{q_{k}}\right) ^{\frac{q_{k-1}}{q_{k}}%
}\right) ^{\frac{q_{k-2}}{q_{k-1}}}\cdots \right) ^{\frac{q_{2}}{q_{3}}%
}\right) ^{\frac{q_{1}}{q_{2}}}\right) ^{\frac{1}{q_{1}}}\leq
C_{k,q_{1}...q_{k}}^{\mathbb{K}}\left\Vert T\right\Vert
\end{equation*}%
}}for all continuous $m$-linear forms $T:c_{0}\times \cdots \times
c_{0}\rightarrow \mathbb{K}$.

(II) $\frac{1}{q_{1}}+\cdots +\frac{1}{q_{k}}\leq \frac{k+1}{2}.$
\end{theorem}

\bigskip

When $k=m$ we recover a characterization of Albuquerque et al. \cite{a}, and
finally when $q_{1}=\cdots =q_{m}=\frac{2m}{m+1}$ we recover the
Bohnenblust--Hille inequality.

\bigskip In this note we present an extension of the above theorem to $%
q_{1},...,q_{m}\in (0,\infty ).$ In particular, we remark that in general
the condition $\frac{1}{q_{1}}+\cdots +\frac{1}{q_{k}}\leq \frac{k+1}{2}$ is
not enough to prove (I) for $q_{1},...,q_{m}\in (0,\infty ).$ For instance,
if $m=k=3$ and $\left( q_{1},q_{2},q_{3}\right) =\left( 1,\frac{18}{10}%
,3\right) $ we have $\frac{1}{q_{1}}+\frac{1}{q_{2}}+\frac{1}{q_{3}}<\frac{%
3+1}{2}$ but, as we shall see, (I) is not true. We prove the following:

\begin{theorem}
\label{999}Let $1\leq k\leq m$ and $n_{1},\ldots ,n_{k}\geq 1$ be positive
integers such that $n_{1}+\cdots +n_{k}=m$, let $q_{1},\dots ,q_{k}\in
\left( 0,\infty \right) $. The following assertions are equivalent:

(i) There is a constant $C_{k,q_{1}...q_{k}}^{\mathbb{K}}\geq 1$ such that%
{\small {\
\begin{equation*}
\left( {\sum\limits_{i_{1}=1}^{\infty }}\left( {\sum\limits_{i_{2}=1}^{%
\infty }}\left( ...\left( {\sum\limits_{i_{k-1}=1}^{\infty }}\left( {%
\sum\limits_{i_{k}=1}^{\infty }}\left\vert T\left( e_{i_{1}}^{n_{1}},\ldots
,e_{i_{k}}^{n_{k}}\right) \right\vert ^{q_{k}}\right) ^{\frac{q_{k-1}}{q_{k}}%
}\right) ^{\frac{q_{k-2}}{q_{k-1}}}\cdots \right) ^{\frac{q_{2}}{q_{3}}%
}\right) ^{\frac{q_{1}}{q_{2}}}\right) ^{\frac{1}{q_{1}}}\leq
C_{k,q_{1}...q_{k}}^{\mathbb{K}}\left\Vert T\right\Vert
\end{equation*}%
}}for all continuous $m$-linear forms $T:c_{0}\times \cdots \times
c_{0}\rightarrow \mathbb{K}$.

(ii) $\sum\limits_{j\in A}\frac{1}{q_{j}}\leq \frac{card(A)+1}{2}$ for all $%
A\subset \{1,...,k\}.$
\end{theorem}

\section{The proof}

\bigskip We begin by recalling the following theorem (in fact, as we
mentioned before, this is precisely Theorem \ref{809} with $k=m$):

\begin{theorem}[Albuquerque, Bayart, Pellegrino and Seoane]
\label{THMBHQ} Let $m\geq 1$, let $q_{1},...,q_{m}\in \lbrack 1,2].$ The
following assertions are equivalent:

(a) There is a constant $C_{q_{1}...q_{m}}^{\mathbb{K}}\geq 1$ such that%
{\small {\
\begin{equation*}
\left( {\sum\limits_{i_{1}=1}^{\infty }}\left( {\sum\limits_{i_{2}=1}^{%
\infty }}\left( ...\left( {\sum\limits_{i_{m-1}=1}^{\infty }}\left( {%
\sum\limits_{i_{m}=1}^{\infty }}\left\vert T\left(
e_{i_{1}},...,e_{i_{m}}\right) \right\vert ^{q_{m}}\right) ^{\frac{q_{m-1}}{%
q_{m}}}\right) ^{\frac{q_{m-2}}{q_{m-1}}}\cdots \right) ^{\frac{q_{2}}{q_{3}}%
}\right) ^{\frac{q_{1}}{q_{2}}}\right) ^{\frac{1}{q_{1}}}\leq C_{q_{1}\ldots
q_{m}}^{\mathbb{K}}\left\Vert T\right\Vert
\end{equation*}%
}}for all continuous $m$-linear forms $T:c_{0}\times \cdots \times
c_{0}\rightarrow \mathbb{K}$.

(b) $\frac{1}{q_{1}}+\cdots +\frac{1}{q_{m}}\leq \frac{m+1}{2}.$
\end{theorem}

\bigskip Our first step is to extend the Theorem \ref{THMBHQ} as follows:

\begin{theorem}
\label{geral} Let $m\geq 1$, let $q_{1},...,q_{m}\in (0,\infty ).$ The
following assertions are equivalent:

(A) There is a constant $C_{q_{1}...q_{m}}^{\mathbb{K}}\geq 1$ such that%
{\small {\
\begin{equation*}
\left( {\sum\limits_{i_{1}=1}^{\infty }}\left( {\sum\limits_{i_{2}=1}^{%
\infty }}\left( ...\left( {\sum\limits_{i_{m-1}=1}^{\infty }}\left( {%
\sum\limits_{i_{m}=1}^{\infty }}\left\vert T\left(
e_{i_{1}},...,e_{i_{m}}\right) \right\vert ^{q_{m}}\right) ^{\frac{q_{m-1}}{%
q_{m}}}\right) ^{\frac{q_{m-2}}{q_{m-1}}}\cdots \right) ^{\frac{q_{2}}{q_{3}}%
}\right) ^{\frac{q_{1}}{q_{2}}}\right) ^{\frac{1}{q_{1}}}\leq C_{q_{1}\ldots
q_{m}}^{\mathbb{K}}\left\Vert T\right\Vert
\end{equation*}%
}}for all continuous $m$-linear forms $T:c_{0}\times \cdots \times
c_{0}\rightarrow \mathbb{K}$.

(B) $\sum\limits_{j\in A}\frac{1}{q_{j}}\leq \frac{card(A)+1}{2}$ for all $%
A\subset \{1,...,m\}.$
\end{theorem}

To prove Theorem \ref{geral} we need to prove the following elementary lemma:

\begin{lemma}
If $q_{1},...,q_{m}\in (0,\infty )$ and
\begin{equation}
\sum\limits_{j\in A}\frac{1}{q_{j}}\leq \frac{card(A)+1}{2}  \label{77}
\end{equation}%
for all $A\subset \{1,...,m\}$, then
\begin{equation*}
\frac{1}{\min \{q_{1},2\}}+\cdots +\frac{1}{\min \{q_{m},2\}}\leq \frac{m+1}{%
2}.
\end{equation*}
\end{lemma}

\begin{proof}
Given $q_{1},...,q_{m}\in (0,\infty )$, consider $A_{1}=\{j:q_{j}\leq 2\}$
and $A_{2}=\{j:q_{j}>2\}$. Then, by (\ref{77}), we have%
\begin{equation*}
\frac{1}{\min \{q_{1},2\}}+\cdots +\frac{1}{\min \{q_{m},2\}}%
=\sum\limits_{j\in A_{1}}\frac{1}{q_{j}}+\sum\limits_{j\in A_{2}}\frac{1}{2}%
\leq \frac{card(A_{1})+1}{2}+\frac{card(A_{2})}{2}=\frac{m+1}{2}.
\end{equation*}
\end{proof}

\bigskip Now we can prove Theorem \ref{geral}. Let us begin by proving that
(A) implies (B). The proof of the case $A=\{1,...,k\}$ for $k<m$ illustrates
the argument. Let
\begin{equation*}
T_{k}:c_{0}\overset{k\text{ times}}{\times \cdots \times }c_{0}\rightarrow
\mathbb{K}
\end{equation*}%
be the $k$-linear form given by the Kahane--Salem--Zygmund inequality (see
\cite[Lemma 6.1]{a}). Then, define%
\begin{equation*}
T_{m}:c_{0}\overset{m\text{ times}}{\times \cdots \times }c_{0}\rightarrow
\mathbb{K}
\end{equation*}%
by%
\begin{equation*}
T_{m}(x^{(1)},...,x^{(m)})=T_{k}(x^{(1)},...,x^{(k)})x_{1}^{(k+1)}...x_{1}^{(m)}.
\end{equation*}%
It is obvious that $\left\Vert T_{m}\right\Vert =\left\Vert T_{k}\right\Vert
$ and by the Kahane--Salem--Zygmund inequality we have $\left\Vert
T_{k}\right\Vert \leq Cn^{\frac{k+1}{2}}$ for some $C>0$.\ Since{\small {\
\begin{eqnarray*}
n^{\frac{1}{q_{1}}+\cdots +\frac{1}{q_{k}}} &=&\left( {\sum%
\limits_{i_{1}=1}^{\infty }}\left( {\sum\limits_{i_{2}=1}^{\infty }}\left(
...\left( {\sum\limits_{i_{k-1}=1}^{\infty }}\left( {\sum\limits_{i_{k}=1}^{%
\infty }}\left\vert T_{k}\left( e_{i_{1}},...,e_{i_{k}}\right) \right\vert
^{q_{k}}\right) ^{\frac{q_{k-1}}{q_{k}}}\right) ^{\frac{q_{k-2}}{q_{k-1}}%
}\cdots \right) ^{\frac{q_{2}}{q_{3}}}\right) ^{\frac{q_{1}}{q_{2}}}\right)
^{\frac{1}{q_{1}}} \\
&=&\left( {\sum\limits_{i_{1}=1}^{\infty }}\left( {\sum\limits_{i_{2}=1}^{%
\infty }}\left( ...\left( {\sum\limits_{i_{m-1}=1}^{\infty }}\left( {%
\sum\limits_{i_{m}=1}^{\infty }}\left\vert T_{m}\left(
e_{i_{1}},...,e_{i_{m}}\right) \right\vert ^{q_{m}}\right) ^{\frac{q_{m-1}}{%
q_{m}}}\right) ^{\frac{q_{m-2}}{q_{m-1}}}\cdots \right) ^{\frac{q_{2}}{q_{3}}%
}\right) ^{\frac{q_{1}}{q_{2}}}\right) ^{\frac{1}{q_{1}}} \\
&\leq &C_{q_{1}\ldots q_{m}}^{\mathbb{K}}\left\Vert T_{m}\right\Vert  \\
&\leq &C_{q_{1}\ldots q_{m}}^{\mathbb{K}}Cn^{\frac{k+1}{2}},
\end{eqnarray*}
}}since $n$ is arbitrary, we conclude that
\begin{equation*}
\sum\limits_{j\in A}\frac{1}{q_{j}}\leq \frac{card(A)+1}{2}.
\end{equation*}

To prove that (B) implies (A), suppose that $\sum\limits_{j\in A}\frac{1}{%
q_{j}}\leq \frac{card(A)+1}{2}$ for all $A\subset \{1,...,m\};$ from the
lemma proved above we have%
\begin{equation*}
\frac{1}{\min \{q_{1},2\}}+\cdots +\frac{1}{\min \{q_{m},2\}}\leq \frac{m+1}{%
2}.
\end{equation*}%
Besides, by considering $A=\{k\}$ it is plain that $q_{k}\geq 1.$ Hence $%
\min \{q_{k},2\}\in \lbrack 1,2]$ for all $k=1,...,m$ and the proof of
Theorem \ref{geral} is now a consequence of (b)$\Rightarrow $(a) and of the
canonical inclusion between the norms of $\ell _{p}$ spaces.

Now we can prove Theorem \ref{999}. Let us begin by proving that (i) implies
(ii). From the proof of the main result of \cite{aa} we know that since (i)
holds for all continuous $m$-linear forms $T_{m}:c_{0}\times \cdots \times
c_{0}\rightarrow \mathbb{K}$, then{\small {\
\begin{equation*}
\left( {\sum\limits_{i_{1}=1}^{\infty }}\left( {\sum\limits_{i_{2}=1}^{%
\infty }}\left( ...\left( {\sum\limits_{i_{k-1}=1}^{\infty }}\left( {%
\sum\limits_{i_{k}=1}^{\infty }}\left\vert T_{k}\left(
e_{i_{1}},...,e_{i_{k}}\right) \right\vert ^{q_{k}}\right) ^{\frac{q_{k-1}}{%
q_{k}}}\right) ^{\frac{q_{k-2}}{q_{k-1}}}\cdots \right) ^{\frac{q_{2}}{q_{3}}%
}\right) ^{\frac{q_{1}}{q_{2}}}\right) ^{\frac{1}{q_{1}}}\leq
C_{k,q_{1}...q_{k}}^{\mathbb{K}}\left\Vert T_{k}\right\Vert
\end{equation*}%
}}holds for all $k$-linear forms $T_{k}:c_{0}\times \cdots \times
c_{0}\rightarrow \mathbb{K}$. $\ $So, from Theorem \ref{geral} we conclude
that%
\begin{equation*}
\sum\limits_{j\in A}\frac{1}{q_{j}}\leq \frac{card(A)+1}{2}
\end{equation*}%
for all $A\subset \{1,...,k\}.$

Now we prove that (ii) implies (i). If $\sum\limits_{j\in A}\frac{1}{q_{j}}%
\leq \frac{card(A)+1}{2}$ for all $A\subset \{1,...,k\}$, from the lemma we
have%
\begin{equation*}
\frac{1}{\min \{q_{1},2\}}+\cdots +\frac{1}{\min \{q_{k},2\}}\leq \frac{k+1}{%
2}.
\end{equation*}%
Since $\min \{q_{j},2\}\in \lbrack 1,2]$ for all $j=1,...,k,$ the proof is
now a consequence of (II)$\Rightarrow $(I) and of the canonical inclusion
between the norms of $\ell _{p}$ spaces.

\end{document}